\documentclass[11pt, letter, oneside, reqno]{article}
\usepackage[total={6.5in, 9in}]{geometry}
\usepackage[noadjust]{cite}
\usepackage{amsmath, amssymb,mathrsfs, dsfont} 
\usepackage{amsmath,amssymb,bm,bbm}
\usepackage{ color, bm,bbm,authblk}
\usepackage{ graphicx,graphics, epstopdf} 
\usepackage{algorithm, algorithmicx,multirow, titlecaps} %
\usepackage{threeparttable}
\usepackage{tikzsymbols}

\DeclareMathAlphabet{\mathbbold}{U}{bbold}{m}{n}

\usepackage[english]{babel}
\usepackage[noend]{algpseudocode} 

\usepackage{tabularx}
\usepackage{booktabs}
\usepackage{multirow}
\usepackage{courier}
\usepackage{float}
\usepackage{makecell}
\setcellgapes{5pt}

\newcommand{\stiny}[1]{{\scalebox{.5}{#1}}}
\newcommand{\smtiny}[1]{{\scalebox{.63}{#1}}}

\newcommand{\argmin}{{\mathrm{argmin}}}

\newcommand*{\argminOp}{\operatornamewithlimits{argmin}\limits}

\newcommand*{\minOp}{\operatornamewithlimits{min}\limits}
\newcommand*{\maxOp}{\operatornamewithlimits{max}\limits}

\newcommand{\tr}{{\smtiny{$\mathsf{T}$ }}\!}

\newcommand{\zero}{\mathbf{0}}

\iftrue
\newcommand{\eye}{\mathbb{I}}
\else
\newcommand{\eye}{\mathbf{I}}
\fi

\newcommand{\vc}[1]{{ \mathrm{#1} }}
\newcommand{\mx}[1]{{ \mathrm{#1} }}

\newcommand{\inner}[2]{{ \langle {#1,#2} \rangle}}

\newcommand{\trace}{\mathrm{tr}}

\newcommand{\Dscr}{{\mathscr{D}}}

\newcommand{\Jcal}{{\mathcal{J}}}

\newcommand{\Ncal}{{\mathcal{N}}}

\newcommand{\Rcal}{{\mathcal{R}}}
\newcommand{\Scal}{{\mathcal{S}}}

\newcommand{\Ucal}{{\mathcal{U}}}
\newcommand{\Vcal}{{\mathcal{V}}}

\newcommand{\Rbb}{{\mathbb{R}}}

\newcommand{\Zbb}{{\mathbb{Z}}}

\usepackage{amsthm}
\renewcommand{\qedsymbol}{$\blacksquare$}

\newtheorem{theorem}{Theorem}
\newtheorem*{theorem*}{Theorem}

\newtheorem*{definition*}{Definition}

\newtheorem{lemma}[theorem]{Lemma}
\newtheorem{remark}{Remark}

\newtheorem*{example*}{Example}

\newtheorem*{claim*}{Claim}

\newtheorem*{problem*}{Problem}

\iftrue 
\usepackage[colorlinks=true]{hyperref}
\usepackage{xcolor}
\hypersetup{
	colorlinks,
	linkcolor={blue!55!black},
	citecolor={red!55!black},
	urlcolor={blue!80!black}
}
\allowdisplaybreaks

\fi
\newcommand{\mxA}{\mx{A}}
\newcommand{\mxB}{\mx{B}}

\newcommand{\mxD}{\mx{D}}
\newcommand{\mxE}{\mx{E}}
\newcommand{\mxK}{\mx{K}}
\newcommand{\mxR}{\mx{R}}

\newcommand{\vca}{\vc{a}}
\newcommand{\vcb}{\vc{b}}
\newcommand{\vcc}{\vc{c}}
\newcommand{\vcd}{\vc{d}}
\newcommand{\vcg}{\vc{g}}
\newcommand{\vcu}{\vc{u}}
\newcommand{\vcq}{\vc{q}}

\newcommand{\vcw}{\vc{w}}
\newcommand{\vcx}{\vc{x}}
\newcommand{\vcy}{\vc{y}}
\newcommand{\vcz}{\vc{z}}

\newcommand{\fro}{\mathrm{F}}

\newcommand{\nS}{n_{\Scal}}

\newcommand{\gLS}{{\vc{g}}^{\text{LS}}}
\newcommand{\gReg}{{\vc{g}}^{\text{Reg}}}
\newcommand{\gKRLS}{{\vc{g}}^{\text{KRLS}}}
\newcommand{\gRLS}{{\vc{g}}^{\text{RLS}}}
\newcommand{\gSRLS}{{\vc{g}}^{\text{SRLS}}}
\newcommand{\gRRegLS}{{\vc{g}}^{\text{RReg}}}
\newcommand{\gSRRegLS}{{\vc{g}}^{\text{SRReg}}}
\newcommand{\gKRRegLS}{{\vc{g}}^{\text{KRReg}}}

\newcommand{\nFIR}{n_{\vc{g}}}
\newcommand{\nD}{n_{\stiny{$\Dscr$}}}

\usepackage{multirow}
\usepackage{booktabs}
\usepackage{array}
\usepackage{lineno}
\title{\LARGE \bf On Robustness of Kernel-Based Regularized System Identification}
\author{Mohammad Khosravi}%
\author{Roy S. Smith}
\affil{Automatic Control Laboratory, ETH Z\"urich 
	\authorcr
	\texttt{\{khosravm,rsmith\}@control.ee.ethz.ch}
}
\date{}
\begin{document}
	\maketitle
\begin{abstract}
This paper presents a novel feature of the kernel-based system identification method.
We prove that the regularized kernel-based approach for the estimation of a finite impulse response is equivalent to a robust least-squares problem with a particular uncertainty set defined in terms of the kernel matrix, and thus, it is called \emph{kernel-based uncertainty set}. We provide a theoretical foundation for the robustness of the kernel-based approach to  input disturbances. Based on robust and regularized least-squares methods, different formulations of system identification are considered, where the kernel-based uncertainty set is employed in some of them. We apply these methods to a case where the input measurements are subject to disturbances. Subsequently, we perform extensive numerical experiments and compare the results to examine the impact of utilizing kernel-based uncertainty sets in the identification procedure. The numerical experiments confirm that the robust least square identification approach with the kernel-based uncertainty set improves the robustness of the estimation to the input disturbances.  
\end{abstract}
\section{Introduction}
The system identification problem, as initially introduced by \cite{zadeh1956identification}, deals with building suitable mathematical models for dynamic systems based on measurement data. Due to the numerous important applications, the subject has received extensive attention \cite{LjungBooK2} developing identification tools and techniques, mainly based on the theory of mathematical statistics and optimization. One of these ideas is employing a suitable regularization to improve the estimation quality by resolving the overfitting issue and also integrating prior knowledge in the estimated model \cite{khosravi2021grad,pillonetto2014kernel, khosravi2021ROA}. The regularization plays this role by introducing a penalty for the model candidates not satisfying specific desired properties. For example, when the model is supposed to have low complexity, the system order is penalized by the rank or the nuclear norm of the Hankel matrix \cite{smith2014frequency}, or based on a similar argument, the atomic norm of a transfer function is utilized in \cite{shah2012linear}.
 
In the seminal work by Pillonetto and De Nicolao \cite{pillonetto2010new}, 
the idea of utilizing \emph{kernel-based} regularization is introduced and led to a paradigm shift in system identification \cite{ljung2020shift}. 
In this approach, the identification problem is framed as a regularized regression where the regularization term is defined based on the norm of a reproducing kernel Hilbert space (RKHS)   \cite{aronszajn1950theory} 
with a suitable structure. The kernel-based approach allows the employment of suitable Tikhonov-like regularizations and  resolves the bias-variance trade-off issue which was not addressed appropriately in previous works \cite{chiuso2016regularization}.
In this framework, the model complexity tuning, the counterpart of model order selection in the classical approaches, is performed efficiently by estimating the hyperparameters determining the kernel and the regularization term \cite{pillonetto2014kernel}.  
Moreover, various types of prior knowledge and desired features of the model can be included in the estimated model by employing suitable forms of the kernel and imposing appropriate constraints on the regression problem. 
The stability and the smoothness of the impulse response can be enforced by utilizing 
kernel functions such as stable splines \cite{pillonetto2014kernel,chen2018kernel}.
By designing kernels based on particular filters, frequency domain attributes such as the time constant and the resonant frequency can be included in the estimated model 
\cite{marconato2016filter}. In \cite{zheng2018positive,khosravi2019positive}, the positivity of the system is addressed by imposing structural constraints in the estimation problem.  The prior knowledge on the DC-gain of the system is considered in \cite{fujimoto2018kernel}. The kernel-based paradigm is extended to the systems with specific structures, e.g., for Hammerstein and Wiener systems \cite{risuleo2017nonparametric,risuleo2019bayesian},  networked systems \cite{ramaswamy2018local}, and periodic systems \cite{yin2020linear}.
Identifying models with low complexity is addressed based on the idea of multi-kernel regularization and sparse hyperparameter selection \cite{chen2014system,khosravi2020low}.
The idea of multiple regularizations together with advanced hyperparameter tuning techniques are used for improving the performance of model estimation \cite{hong2018multiple-SURE, chen2018regularized, mu2018asymptotic-GCV, khosravi2020regularized}. Recently, it has been empirical observed that kernel-based approaches can improve the robustness of the estimation to the input disturbances \cite{hiroe2020kernel}. 

Inspired by \cite{el1997robust}, we introduce a new aspect of kernel-based identification in this paper. More precisely, we show that the kernel-based regularized estimation of a finite impulse response is equivalent to a robust least-squares problem. This result provides a theoretical foundation for the robustness of the kernel-based approach with respect to the input disturbances. The uncertainty set obtained in the robust optimization problem has an interesting special shape and structure which is defined based on the regularization matrix. Accordingly, the set is called, the  \emph{kernel-based uncertainty set}. In order to study the nature of this set, we consider various identification approaches formulated  in terms of robust and regularized least-squares, and then apply these methods to an estimation case in which the input measurements are subject to disturbances. By means of Monte Carlo experiments, we compare these approaches to examine the impact of employing kernel-based uncertainty sets. The numerical experiments show that a robust least-squares approach with a kernel-based uncertainty set, not only inherits the interesting features of the kernel-based identification approach, but also improves the robustness of the estimation with respect to the input disturbances.

\section{Notation}\label{sec:Not}
The set of integers, the set of non-negative integer numbers, the set of real numbers, and the set of non-negative real numbers are denoted by $\Zbb$, $\Zbb_+$, $\Rbb$, and $\Rbb_+$, respectively.
The $n$-dimensional Euclidean space is denoted by $\Rbb^n$ and the set of $n$ by $m$ matrices is $\Rbb^{n\times m}$.
The zero vector is denoted by  $\zero$ and the identity matrix is denoted by $\eye$.
For a matrix $\mx{A}\in\Rbb^n$, $\trace(A)$ denotes the trace of $\mxA$.
For a non-singular matrix $\mxA$, the transpose of $\mxA^{-1}$ is $\mxA^{-\tr}$. 
Given a positive-definite matrix  $\mxK\in\Rbb^{n\times n}$,  we define an inner product on $\Rbb^n$ as $\inner{\vca}{\vcb}_{\mxK}:=\vca^\tr \mxK\vcb$, for any $\vca,\vcb\in\Rbb^n$. 
Subsequently, an induced norm, denoted by $\|\cdot\|_{\mxK}$, is defined as $\|\vca\|_{\mxK}:=(\vca^\tr \mxK\vca)^{\frac{1}{2}}$. When $\mxK=\eye$, we have the Euclidean norm which is denoted by $\|\cdot\|$. 
Similarly, we can define an inner product on $\Rbb^{m\times n}$  denoted by $\inner{\cdot}{\cdot}_{\mxK}$ and defined as $\inner{\mxA}{\mxB}_{\mxK}:=\trace(\mxA \mxK\mxB^\tr)$, for any $\mxA,\mxB\in\Rbb^{m\times n}$. 
The corresponding induced norm is shown by $\|\!\cdot\!\|_{\mxK}$. When $\mxK=\eye$, we have the Frobenius norm denoted by $\|\!\cdot\!\|_{\fro}$.
The expression $X\sim \Ncal(\mu,\Sigma)$ says that random variable $X$ has a Gaussian distribution with mean $\mu$ and covariance $\Sigma$. 
We denote the probability density and the conditional probability density by $p(\cdot)$ and $p(\cdot|\cdot)$,  respectively.

\section{Regularized System Identification}
Let $\Scal$ be a discrete-time single-input-single-output stable and causal linear time-invariant (LTI) system, with a transfer function  $G_{\Scal}(\vc{q})$ defined as
\begin{equation}
\begin{array}{c}	
G_{\Scal}(\vc{q}) := \sum_{k=0}^{\infty} g_k \vc{q}^k,
\end{array}
\end{equation}
where $\vcg_{\smtiny{$\Scal$}}:=(g_k)_{k\in\Zbb_+}$ denotes the impulse response of the system and $\vc{q}$ denotes the {\em forward shift operator}. 
Due to the stability of $\Scal$,
we know that $\vcg\in\ell_1(\Zbb_+)$, i.e.,
\begin{equation}
\sum_{k=0}^\infty |g_k| <\infty.	
\end{equation}
Therefore, for any $\epsilon>0$, there exists $\nS(\epsilon)\in\Zbb_+$ such that
$\sum_{k=\nS(\epsilon)}^\infty |g_k| <\epsilon$.
Accordingly, the infinite impulse response (IIR) can be truncated at a sufficiently high order $\nFIR\in \Zbb_+$ to approximate the system with a finite-length impulse response (FIR) denoted by $\vcg$ and defined as 
\begin{equation}
\vcg:=
\begin{bmatrix}
g_0&g_1&g_2&\ldots&g_{\nFIR-1}
\end{bmatrix}^\tr\in \Rbb^{\nFIR}.
\end{equation}
Thus, the finite impulse response (FIR) model of the system is as 
\begin{equation}
G(\vc{q}) := \sum_{k=0}^{\nFIR-1} g_k \vc{q}^k.
\end{equation}
\subsection{Identification Problem: FIR Formulation}
Let $\vcu:=(u_t)_{t\in \Zbb}$ be a bounded input signal given to the system $\Scal$. 
Also, let $\vcy:=(y_t)_{t\in\Zbb}$ be the corresponding output signal which is subject to measurement uncertainty. 
Accordingly, we have
\begin{equation}
\label{eqn:y_t=sum_kg_ku_t-k}
y_t = \sum_{k=0}^{\nFIR-1}g_k u_{t-k} + w_t, \qquad \forall t\in \Zbb,
\end{equation}
where $w_t$ corresponds to the measurement noise and the unmodeled part of the system. 
Note that when the output measurement is subject to bounded uncertainty, then due to the stability of the system $\vcy$, is a bounded signal.
For time instants $t=0,1,\ldots,\nD\!-\!1$, let assume the inputs and the measured outputs of the system are given. We define the set of data, denoted by $\Dscr$, as
\begin{equation}
\Dscr 
:= \Big\{(u_t,y_t)\ \Big| \ t=0,1,\ldots,\nD\!-\!1\Big\}.
\end{equation}
Following this, the identification problem is defined as estimating the FIR model of the system, $\vcg$, using $\Dscr$.

\subsection{Prediction Error Method}\label{sec:LS_or_PEM}
Let vector $\varphi_t$ be defined as
\begin{equation}
\varphi_t := 
\begin{bmatrix}
u_t & u_{t-1} & \ldots & u_{t-\nFIR+1}\\
\end{bmatrix}^\tr\in\Rbb^{\nFIR},
\end{equation}
for any $t\in\Zbb$.
For a candidate FIR model $\vcg$, one can introduce the one-step ahead prediction rule for the output at time instant $t$ as $\hat{y}(t|\vcg):=\varphi_t^\tr \vcg$.
Considering the set of data $\Dscr$, the quality of the prediction rule can be assessed by comparing the actual measured outputs with the predicted values. To this end, one can form an empirical loss, $\Vcal_{\Dscr}:\Rbb^{\nFIR}\to\Rbb$, for evaluating the prediction quality over set of data $\Dscr$, e.g., $\Vcal_{\Dscr}$ can be defined as sum of squared errors of predictions, i.e., one has
\begin{equation}\label{eqn:V_sum_quared_err}
\Vcal_{\Dscr}(\vcg) := \sum_{t=0}^{\nD-1}(y_t - \hat{y}(t|\vcg))^2.
\end{equation}
Then, one can estimate the FIR model by minimizing the empirical loss $\Vcal_{\Dscr}$. Define vectors $\vcy$ and $\vcw$ respectively  as
\begin{equation}\label{eqn:y}
\begin{split}
\vc{y} :=
\begin{bmatrix}
y_0&\cdots&y_{\nD-1}
\end{bmatrix}^\tr \in \Rbb^{\nD},
\end{split}
\end{equation}
and 
\begin{equation}\label{eqn:w}
\begin{split}
\vc{w} :=
\begin{bmatrix}
w_0&\cdots&w_{\nD-1}
\end{bmatrix}^\tr \in \Rbb^{\nD}.
\end{split}
\end{equation}
Then, due to \eqref{eqn:y_t=sum_kg_ku_t-k}, 
one can easily see that
\begin{equation}
	\vc{y} = \Phi \vc{g} + \vc{w},
\end{equation} 
where $\Phi\in \Rbb^{\nD\times \nFIR}$ is a Toeplitz matrix defined as follows
\begin{equation}\label{eqn:Phi_teoplitz}
\Phi := 
\begin{bmatrix}
\varphi_0^\tr\\
\varphi_1^\tr\\
\vdots\\
\varphi_{\nD-1}^\tr\\
\end{bmatrix}
=
\begin{bmatrix}
u_0&u_{-1}&\ldots&u_{-\nFIR+1}\\
u_1&u_{0} &\ldots&u_{-\nFIR+2}\\
\vdots&\vdots&\ddots&\vdots\\
u_{\nD-1}&u_{\nD-2}&\ldots&u_{\nD-\nFIR}\\
\end{bmatrix}.
\end{equation}
Consequently, regarding the empirical loss, we have $\Vcal_{\Dscr}(\vcg) = \|\vcy-\Phi\vcg\|^2$.
Here, for the sake of simplicity, one can assume that the system is initially at rest, i.e., $u_t=0$, for all $t<0$, and thus, given set of data $\Dscr$, matrix $\Phi$ is known. 
Accordingly, estimating $\vcg$ by minimizing the empirical loss is well defined and leads to the following least-squares (LS) 
\begin{equation}\label{eqn:LS}
\begin{split}
\gLS 
&:= \argmin_{\vc{g}\in \Rbb^{\nFIR}}\ \|\vc{y} - \Phi\vc{g}\|^2
\\&= (\Phi^\tr\Phi)^{-1}\Phi^\tr\vc{y},
\end{split}
\end{equation}
where the last equality holds when $\Phi^\tr\Phi$ is non-singular.
When $\Phi^\tr\Phi$ is not full-rank, the solution is not unique
and the estimation problem is ill-posed. This issue happens for example when $\nD$ is small. If the condition number of $\Phi^\tr\Phi$ is high, then the estimation $\gLS$ is sensitive to noise. 
These issues can be resolved by including an appropriate regularization in \eqref{eqn:LS} \cite{pillonetto2014kernel, chen2018kernel}.

\subsection{Regularization Method}\label{sec:RegLS_method}
Let $\mxK$ be a positive definite matrix. Define the regularization function $\Rcal:\Rbb^{\nFIR}\to \Rbb$ as  $\Rcal(\vcg) = \vcg^\tr\mxK^{-1}\vcg$, for any $\vcg\in\Rbb^{\nFIR}$.
In order to introduce the regularized estimation problem, we add this term to the empirical loss and obtain a regularized loss function $\Jcal:\Rbb^{\nFIR}\to \Rbb$ as
\begin{equation}\label{eqn:J_general}
\Jcal(\vcg) := \Vcal_{\Dscr}(\vcg)+\lambda \Rcal(\vcg), \qquad \forall \vcg\in\Rbb^{\nFIR},
\end{equation}
where $\lambda>0$ is the regularization weight.

Let $\Vcal_{\Dscr}$ be defined as in \eqref{eqn:V_sum_quared_err}. Then, the regularized estimation cost in \eqref{eqn:J_general} can be written as
\begin{equation}\label{eqn:J_RegLS}
\Jcal(\vcg) := \|\vcy-\Phi \vcg\|^2+\lambda \vcg^\tr\mxK^{-1}\vcg.
\end{equation}
Then, the regularized estimation of the FIR is 
\begin{equation}\label{eqn:RegLS}
\begin{split}
\gReg 
& := \argminOp_{\vc{g}\in\Rbb^{\nFIR}}\ 
\|\vcy-\Phi \vcg\|^2+\lambda \vcg^\tr\mxK^{-1}\vcg,\\
&= (\Phi^\tr\Phi + \lambda \mxK^{-1})\Phi^\tr\vcy. 
\end{split}
\end{equation}

The regularization matrix $\mxK$, also known as the {\em kernel matrix}, 
integrates in to the estimated impulse response additionally available prior information and desired attributes such as stability and smoothness. 
Additionally, it can reduce the variance of the estimation and resolve 
the potential ill-posedness of the problem and the bias-variance trade-off issue. 
In the literature, various methods are introduced for determining the regularization matrix $\mxK$ \cite{pillonetto2014kernel,zorzi2018harmonic,chen2018kernel,marconato2016filter}, e.g., by employing common kernels such as  {\em tuned/correlated} (TC) and  {\em stable spline} (SS) \cite{pillonetto2014kernel}, or designing $\mxK$ can be based on multiple regularizations \cite{hong2018multiple-SURE,chen2014system,chen2018regularized,khosravi2020low,khosravi2020regularized}.

\section{Robustness of Kernel-Based Regularized System Identification}
\label{sec:Robustness_RegLS}
In this section, we show that the kernel-based regularized identification is equivalent to a robust least-squares estimation with a specific structure. 
\begin{theorem}\label{thm:RegLS_is_KRLS}
There exist $\rho>0$ such that
the regularized estimation of the impulse response derived in \eqref{eqn:RegLS} is the solution of the following robust least-squares problem
\begin{equation}\label{eqn:thm_robust_LS_1}
	\minOp_{\vcg\in\Rbb^{\nFIR}}\ 
	\maxOp_{\substack{\Delta\in \Rbb^{\nD\times \nFIR}\\ \|\Delta\|_{\mxK}\le \rho}}\ 
	\|(\Phi+\Delta)\vcg-\vcy\|.
\end{equation}
\end{theorem}
\begin{proof}
Let $\mxR$ be the square root of $\mxK$ or the lower triangular matrix in the Cholesky decomposition of $\mxK$. Then, the inner problem in 	\eqref{eqn:thm_robust_LS_1} can be shown as to be
\begin{equation}\label{eqn:thm_robust_LS_2}
\maxOp_{\substack{\Delta\in \Rbb^{\nD\times \nFIR}\\ \|\Delta\mxR\|_{\fro}\le \rho}}\ 
\|(\Phi+\Delta)\vcg-\vcy\|^2.
\end{equation}
Due to Lemma \ref{lem:inner_problem_robust} (see Appendix \ref{sec:Lemma_appendix}), 
the optimal value of \eqref{eqn:thm_robust_LS_2} equals  $\|\Phi\vcg-\vcy\| + \|\rho\mxR^{-1}\vcx\|$.
Accordingly, the inner problem is equivalent to the following convex program
\begin{equation}\label{eqn:thm_robust_LS_3}
\begin{split}
\minOp_{a,b} 
& \quad
a
\\
\mathrm{s.t.}
& \quad
\|\Phi\vcg-\vcy\| \le a-b,\ \ 
\\
& \quad
\|\rho\mxR^{-1}\vcg\|\le b.
\end{split} 
\end{equation}
Consequently, the robust optimization problem \eqref{eqn:thm_robust_LS_1} can be written as in the following equivalent form
\begin{equation}\label{eqn:thm_robust_LS_4}
\begin{split}
\minOp_{\substack{\vcg\in\Rbb^{\nFIR}\\a,b\in\Rbb} }
& \quad
a
\\
\mathrm{s.t.}\ \!\ \quad
& \quad
\|\Phi\vcg-\vcy\| \le a-b,\ \ 
\\& \quad
\|\rho\mxR^{-1}\vcg\|\le b.
\end{split} 
\end{equation} 
Note that at the optimal solution $(\vcg^*,a^*,b^*)$, we have 
$a^*=\|\Phi\vcg^*-\vcy\|+\rho\|\mxR^{-1}\vcg^*\|$, where $a^*$ is the optimal value of \eqref{eqn:thm_robust_LS_4} as well.
Defining the vector $\vcx$ as $\vcx:=\begin{bmatrix}\vcg^\tr&a&b\end{bmatrix}^\tr$, we can re-write the problem \eqref{eqn:thm_robust_LS_4} in the following form
\begin{equation}\label{eqn:thm_robust_LS_5}
\begin{split}
\minOp_{\vcx\in \Rbb^{\nFIR+2}}
& \quad
\begin{bmatrix}\zero^\tr&1&0\end{bmatrix}\vcx
\\
\mathrm{s.t.}\ \
& \quad
\|\begin{bmatrix}\Phi&\zero&\zero\end{bmatrix} \vcx-\vcy\| 
\le 
\begin{bmatrix}\zero^\tr&1&-1\end{bmatrix}\vcx,
\\
& \quad
\|\begin{bmatrix}\rho\mxR^{-1}&\zero&\zero\end{bmatrix} \vcx\| 
\le \begin{bmatrix}\zero^\tr&0&1\end{bmatrix} \vcx.\\
\end{split}
\end{equation}
If $\epsilon>0$, then $\vcx = \begin{bmatrix}\zero^\tr&\|\vcy\|+1+\epsilon&1\end{bmatrix}^\tr$ is strictly feasible  for \eqref{eqn:thm_robust_LS_5}, and also, the problem is bounded. Therefore, the Slater's conditions hold.
The dual of \eqref{eqn:thm_robust_LS_5} is as follows
\begin{equation*}\label{eqn:thm_robust_LS_6}
\begin{split}
\maxOp_{s,t,\vcz,\vcw} 
& \quad
\vcy^\tr \vcz\\
\mathrm{s.t.}\
& \quad\!\!
\begin{bmatrix}\Phi^\tr\\\zero^\tr\\\zero^\tr\end{bmatrix}\vcz
+
\begin{bmatrix}\rho\mxR^{-\tr}\\\zero^\tr\\\zero^\tr\end{bmatrix}\vcw +
\begin{bmatrix}\zero\\1\\-1\end{bmatrix}s
+
\begin{bmatrix}\zero\\0\\1\end{bmatrix}
t=
\begin{bmatrix}\zero\\1\\0\end{bmatrix},
\\&\quad  \|\vcz\|\le s, 
\\&\quad  \|\vcw\|\le t,
\\&\quad  \vcz\in\Rbb^{\nD},\vcw\in\Rbb^{\nFIR}, s,t\in\Rbb,
\end{split}
\end{equation*}
which simplifies to the following optimization problem
\begin{equation}\label{eqn:thm_robust_LS_7}
\begin{split}
\maxOp_{\substack{\vcz\in\Rbb^{\nD}\\\vcw\in\Rbb^{\nFIR}} }
&\ 
\vcy^\tr \vcz\\
\mathrm{s.t.}\ \ \ \ \
&\ \Phi^\tr \vcz + \rho\mxR^{-\tr}\vcw = \zero,
\\& \|\vcz\|\le 1,
\\& \|\vcw\| \le  1.
\end{split}
\end{equation}
Note that \eqref{eqn:thm_robust_LS_7} is feasible and bounded. 
Let $(\vcz^*,\vcw^*)$ be the optimal solution for the dual problem \eqref{eqn:thm_robust_LS_7}. 
Due to strong duality, we know that 
$a^* = \vcy^\tr\vcz^*$. 
Accordingly, for the optimal solutions of \eqref{eqn:thm_robust_LS_5} and \eqref{eqn:thm_robust_LS_7}, we have that
\begin{equation}\label{eqn:thm_robust_LS_8}
\begin{split}
\|\Phi\vcg^*-\vcy\|+\rho\|\mxR^{-1}\vcg^*\|
&= 
\vcy^\tr\vcz^*
\\&=
(\vcy-\Phi\vcg^*)^\tr\vcz^*
+\vcg^*{}^\tr\Phi^\tr\vcz^*
\\&=
(\vcy-\Phi\vcg^*)^\tr\vcz^*
-\rho\vcg^*{}^\tr\mxR^{-\tr}\vcw^*.
\end{split}	
\end{equation}
Since $\|\vcz^*\|,\|\vcw^*\|\le 1$, due to Cauchy-Schwartz inequality, \eqref{eqn:thm_robust_LS_8} holds if and only if
\begin{equation}
\vcz^* = \frac{\vcy-\Phi\vcg^*}{\|\vcy-\Phi\vcg^*\|}, \qquad
\vcw^* = - 
\frac{\mxR^{-1}\vcg^*}{\|\mxR^{-1}\vcg^*\|}.
\end{equation} 
Following this and due to the equality constraint in \eqref{eqn:thm_robust_LS_7}, one can see that 
\begin{equation}\label{eqn:thm_robust_LS_10}
\begin{split}
\zero &= 
\Phi^\tr\vcz^*+\rho\mxR^{-\tr}\vcw^*	
\\&=
\Phi^\tr\frac{\vcy-\Phi\vcg^*}{\|\vcy-\Phi\vcg^*\|}
-
\rho\mxR^{-\tr}
\frac{\mxR^{-1}\vcg^*}{\|\mxR^{-1}\vcg^*\|}
\\&
=
\frac{\Phi^\tr\vcy-\Phi^\tr\Phi\vcg^*}{\|\vcy-\Phi\vcg^*\|}
-
\frac{\rho\mxK^{-1}\vcg^*}{\|\mxR^{-1}\vcg^*\|}.
\end{split}	
\end{equation}
Rearranging the terms in \eqref{eqn:thm_robust_LS_10}, one has
\begin{equation}\label{eqn:thm_robust_LS_11}
\begin{split}
\Phi^\tr\Phi\vcg^*
+
\frac{\rho\|\vcy-\Phi\vcg^*\|}{\|\mxR^{-1}\vcg^*\|}\mxK^{-1}\vcg^* &= \Phi^\tr\vcy.
\end{split}	
\end{equation}
Subsequently, it follows that
\begin{equation}\label{eqn:thm_robust_LS_12}
\vcg^* = \Big(\Phi^\tr\Phi
+
\mu\mxK^{-1}\Big)^{-1}\Phi^\tr\vcy.
\end{equation}
where  $\mu = \frac{\rho\|\vcy-\Phi\vcg^*\|}{\|\mxR^{-1}\vcg^*\|}$.
By choosing $\rho$ appropriately, one can set $\mu = \lambda $. Consequently, due to \eqref{eqn:RegLS}, we have $\vcg^* = \gReg$. This concludes the proof.
\end{proof}
\begin{remark}
From the proof of Theorem \ref{thm:RegLS_is_KRLS}, one can see that
\begin{equation}
\rho := 
\lambda\frac{\|\mxR^{-1}\gReg\|}
{\|\Phi\gReg-\vcy\|}
=
\lambda\frac{\left(\gReg{}^\tr\mxK^{-1}\gReg\right)^{\frac{1}{2}}}
{\|\Phi\gReg-\vcy\|}.
\end{equation}
Accordingly, when the model fits well to the data, the residuals as well as $\|\Phi\gReg-\vcy\|$ are small. Subsequently $\rho$ has a large value and the estimation is robust.
\end{remark}
\begin{remark}
The main ingredient in the definition of the optimization problem \eqref{eqn:thm_robust_LS_1} is the kernel matrix $\mxK$. Accordingly, we call this method \emph{kernel-based robust least-squares}.
\end{remark}

\section{Kernel-Based Uncertainty Set}
In this section, we further study the \emph{kernel-based uncertainty set} introduced in Section \ref{sec:Robustness_RegLS}. 

Let $\Ucal_{\rho}$ denote the uncertainty set employed in \eqref{eqn:thm_robust_LS_1}, i.e., $\Ucal_{\rho}$ is defined as
\begin{equation}\label{eqn:Urho}
	\Ucal_{\rho} := \Big\{\Delta\in\Rbb^{\nD\times\nFIR}
	\ \Big|\
	\|\Delta\|_{\mxK}^2 = \trace(\Delta\mxK\Delta^\tr)\le \rho^2\Big\}.
\end{equation}
One can see that $\Ucal_{\rho}$ is a hyperball in the inner product space $(\Rbb^{\nD\times\nFIR},\inner{\cdot}{\cdot}_{\mxK})$ which is  centered at the origin and has radius $\rho$.
The inner product $\inner{\cdot}{\cdot}_{\mxK}$ is defined based on the kernel matrix $\mxK$ and determines the geometry of the space as well as the shape and the structure of  $\Ucal_{\rho}$.
 
\begin{figure}[t]
\begin{center}
\includegraphics[width=0.35\textwidth]{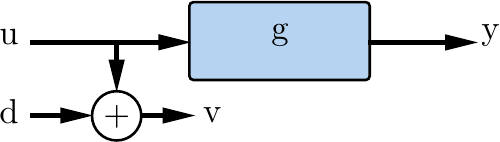}
\end{center}
\caption{System with disturbance in input measurements.}
\label{fig:BD01}
\end{figure}
In order to investigate the nature of the uncertainty set $\Ucal_{\rho}$, we consider the identification problem where the measurements of the input signal are subject to disturbances, as shown in Figure \ref{fig:BD01}. More precisely, at time instant $t$, the measured input is $v_t = u_t+d_t$, where $u_t$ is the value of true input and $d_t$ is the value of measurement disturbance, and also, let $y_t=\sum_{k=0}^{\nFIR-1}g_ku_{t-k}$ be the noiseless measured output of the system, for any $t=0,\ldots,\nD\!-\!1$. 
Let $\Psi,\Delta\in\Rbb^{\nFIR\times\nD}$ be Toeplitz matrices defined similarly to $\Phi$
as in \eqref{eqn:Phi_teoplitz} based on the measured inputs, disturbances and the unknown history of the input signal. Therefore, we have $\Psi = \Phi+\Delta$ and $\vcy=\Phi\vcg$, and subsequently, one can write $\vcy$ in two equivalent forms,  
\begin{equation}\label{eqn:yPsiDeltag}
\vcy=(\Psi-\Delta)\vcg,
\end{equation}
and 
\begin{equation}\label{eqn:yPhigDeltag}
\vcy=\Psi\vcg + (-\Delta\vcg).
\end{equation}
Motivated by \eqref{eqn:yPsiDeltag} and \eqref{eqn:yPhigDeltag}, one may propose variations on robust and regularized least-squares methods for estimating $\vcg$.
In the following, we provide the details of some these formulations.
\\1) Ordinary Least-Squares (LS): Motivated by \eqref{eqn:yPhigDeltag} and based on a discussion similar to Section \ref{sec:LS_or_PEM}, one may propose a least-squares approach to estimate $\vcg$ as following
\begin{equation}\label{eqn:LS_Psi}
	\gLS := \argminOp_{\vcg\in \Rbb^{\nFIR}}\|\vcy-\Psi \vcg\|^2.
\end{equation}
\\2)
Regularized Least-Squares (RegLS): Similar to the least-squares case and the arguments provided in Section \ref{sec:RegLS_method}, a regularized least-squares approach can be proposed for estimating $\vcg$ as following
\begin{equation}\label{eqn:RegLS_Psi}
\gReg := \argminOp_{\vc{g}\in\Rbb^{\nFIR}}\ 
\|\vcy-\Psi \vcg\|^2+\lambda \vcg^\tr\mxK^{-1}\vcg.
\end{equation}
\\3)
Robust Least-Squares (RLS): Considering \eqref{eqn:yPsiDeltag} and \eqref{eqn:yPhigDeltag} along with the least-squares discussion, 
one may propose an estimation approach formulated as a standard robust least-squares problem as following
\begin{equation}\label{eqn:RLS_Psi}
	\gRLS:=
	\argminOp_{\vc{g}\in\Rbb^{\nFIR}}\Bigg[ 
	\maxOp_{\substack{\Delta\in \Rbb^{\nD\times \nFIR}\\ \|\Delta\|_{\fro}\le \rho}} 
	\|\vcy-(\Psi-\Delta)\vcg\|^2\Bigg].
\end{equation}
\\4)
Structured Robust Least-Squares (SRLS): Since the uncertainty matrix $\Delta$ in \eqref{eqn:yPsiDeltag} has a Toeplitz structure, we may 
formulate a structured robust least-squares for the estimation of $\vcg$. 
To this end, given $\rho\in\Rbb_+$, we define the \emph{structured uncertainty set}, $\Scal_{\rho}$,  as
\begin{equation*}\label{eqn:Srho}
\Scal_{\rho}:=\!\!  \Bigg\{
\sum_{k=-\nFIR+1}^{\nD-1}\!\! \delta_k \mxE^{(k)}
\Big|
\sum_{k=-\nFIR+1}^{\nD-1}\!\!  \delta_k^2\le \rho
\Bigg\}, 
\end{equation*}
where for each $k=-\nFIR+1,\ldots,\nD-1$,  $\mxE^{(k)}\in \Rbb^{\nD\times\nFIR}$ is a Toeplitz matrix with entries in $\{0,1\}$ such that  $\mxE^{(k)}_{i,j}=1$ only when $i-j=k$.
Then, the estimation problem is 
\begin{equation}\label{eqn:SRLS_Psi}
\gSRLS:=
\argminOp_{\vc{g}\in\Rbb^{\nFIR}}\Bigg[ 
\maxOp_{\Delta\in \Scal_{\rho}} 
\|\vcy-(\Psi-\Delta)\vcg\|^2\Bigg].
\end{equation}
\\5)
Kernel-based Robust Least-Squares (KRLS): According to \eqref{eqn:yPsiDeltag} and the discussion in Section \ref{sec:Robustness_RegLS} introducing the kernel-based uncertainty set $\Ucal_{\rho}$ (see equation \eqref{eqn:Urho}), we can formulate an estimation approach for $\vcg$ as the following robust least-squares problem
\begin{equation}\label{eqn:KRLS_Psi}
	\gKRLS:=
	\argminOp_{\vc{g}\in \Rbb^{\nFIR}}\Bigg[ 
	\maxOp_{\Delta\!\ \in\!\ \Ucal_{\rho}} 
	\|\vcy-(\Psi-\Delta)\vcg\|^2\Bigg].
\end{equation}
\\6) 
Robust Regularized Least-Squares (RRegLS): By including regularization in RLS method, we obtain a robust regularized least-squares estimation approach as follows,
\begin{equation}\label{eqn:RRegLS_Psi}
\begin{split}	
&\gRRegLS:=
\argminOp_{\vc{g}\in\Rbb^{\nFIR}}\Bigg[ 
\maxOp_{\substack{\Delta\in \Rbb^{\nD\times \nFIR}\\ \|\Delta\|_{\fro}\le \rho}}
\|\vcy-(\Psi\!-\!\Delta)\vcg\|^2+\lambda \vcg^\tr\mxK^{-1}\vcg\Bigg].
\end{split}
\end{equation}
\\7)
Structured Robust Regularized Least-Squares (SRRegLS): Similar to the previous case, by considering the regularization in SRLS approach, one can obtain a structured robust regularized least-squares estimation as the following optimization problem
\begin{equation}\label{eqn:SRRegLS_Psi}
\begin{split}	
&\gSRRegLS:=
\argminOp_{\vc{g}\in\Rbb^{\nFIR}}\Bigg[
\maxOp_{\Delta\in \Scal_{\rho}} 
\|\vcy-(\Psi-\Delta)\vcg\|^2+\lambda \vcg^\tr\mxK^{-1}\vcg\Bigg].
\end{split}
\end{equation}
\\8)
Kernel-based Robust Regularized Least-Squares (KRRegLS): One can add the regularization term in KRLS approach, and subsequently, obtain a kernel-based robust regularized least-squares method for the estimation of $\vcg$ as the following optimization problem
\begin{equation}\label{eqn:KRRegLS_Psi}
\begin{split}	
&\gKRRegLS:=
\argminOp_{\vc{g}\in\Rbb^{\nFIR}}\Bigg[
\maxOp_{\Delta\!\ \in\!\ \Ucal_{\rho}} 
\|\vcy-(\Psi-\Delta)\vcg\|^2+\lambda \vcg^\tr\mxK^{-1}\vcg\Bigg].
\end{split}
\end{equation}
\begin{remark}
Each of the above optimization problems 
is either convex or can be reformulated as a convex program.
\end{remark}

\begin{figure}[t]
	\begin{center}
		\includegraphics[width=0.55\textwidth]{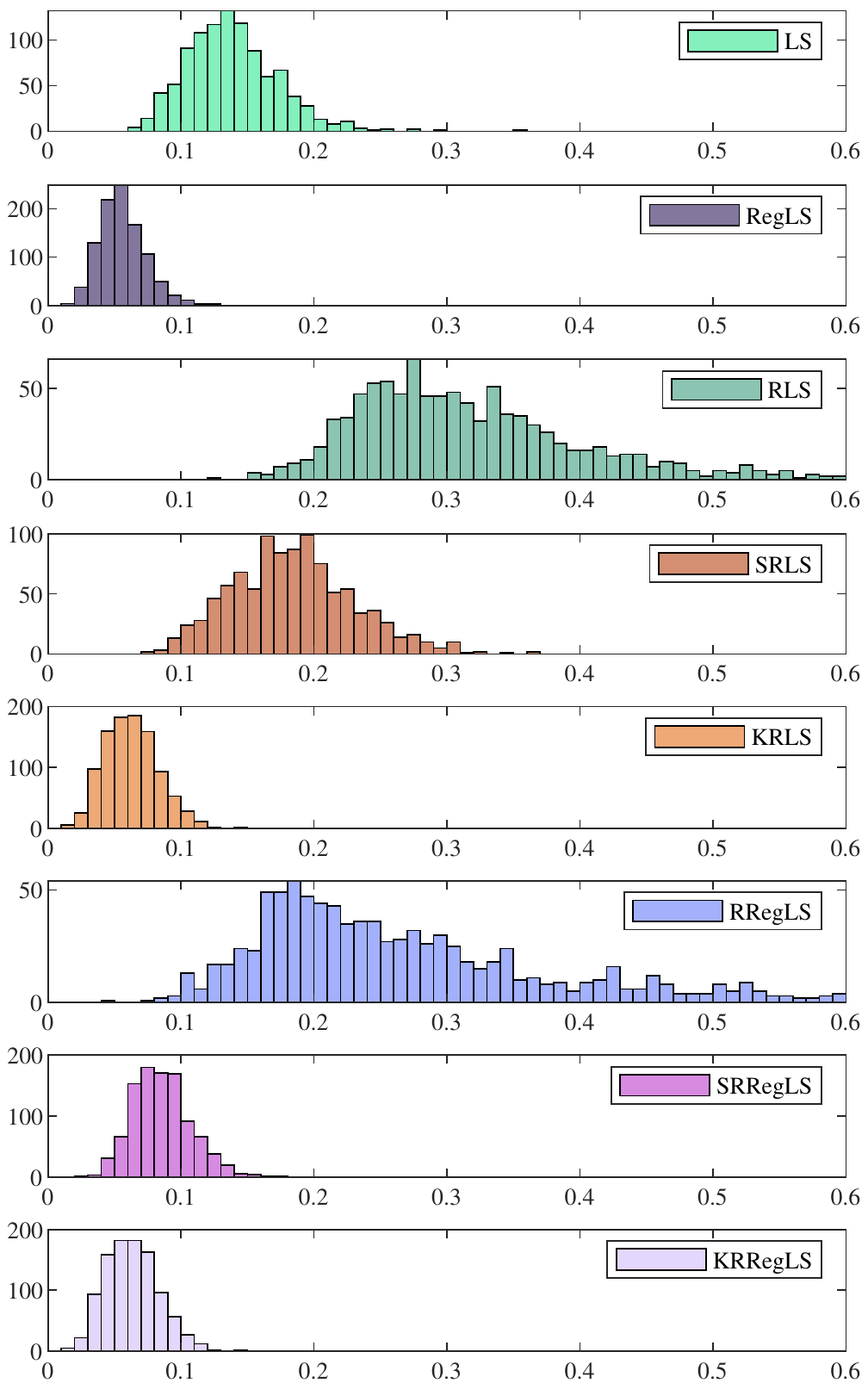}
	\end{center}
	\caption{Histograms of normalized RMSE for different estimation methods.}
	\label{fig:hist1}
\end{figure}

The differences between these formulations  depend on the choice of uncertainty set and also whether a kernel-based regularization is employed or not. 
In order to investigate the impact of the  kernel-based uncertainty set, employed in KRLS and KRRegLS, we compare the performance of these approaches by means of a Monte Carlo numerical experiment.
To this end, we consider the following system \cite{ljung2020shift},
\begin{equation}
\mx{G}(\vcq) = \frac{0.02008 + 0.04017\vcq^{-1} + 0.02008\vcq^{-2}}{1 - 1.561\vcq^{-1} + 0.6414\vcq^{-2}},
\end{equation}
and perform 1000 experiment runs.
In each of these experiments, the system is simulated with a PRBS signal of length $\nD=127$ and the input is measured with measurement disturbance $d_t\sim\Ncal(0,\sigma_d^2)$ where $\sigma_d=0.1$, for $t=0,\ldots,\nD\!-\!1$. The output measurement is taken to be noiseless. 
In order to identify the system, we approximate $\mx{G}$ with a FIR $\vcg$ of length $\nFIR=80$, and then apply each of the estimation methods formulated above.
The value of $\rho$ is calculated based on the true disturbances.

Figure \ref{fig:hist1} shows the histograms of normalized root mean squared errors (RMSE), $\frac{\|\hat{\vcg}-\vcg\|}{\|\vcg\|}$, for each of the above methods. 
For all of the histograms, the x-axis range is taken as $[0,0.6]$ to visualize the results better, however, the RLS and RRegLS have  values above $0.6$. In order to evaluate the quality of estimation, we employ R-squared metric defined as follows
\begin{equation}
\mathrm{R^2}(\hat{\vcg}):= 100 \times \bigg[1- \frac{\|\hat{\vcg}-\vcg\|}{\|\vcg-\bar{\vcg}\|}\bigg],
\end{equation}
where $\bar{\vcg}\in\Rbb^{\nFIR}$ is a vector such that each of its entries equals to $\frac{1}{\nFIR}\sum_{k=0}^{\nFIR-1}g_k$, i.e., the average of entries in $\vcg$. 
Figure \ref{fig:box1} demonstrates and compares the fitting results.
The values of bias, variance and MSE for each of the above estimation methods are provided in Table \ref{tbl:mc}. 

\begin{figure}[t]
	\begin{center}
		\includegraphics[width=0.65\textwidth]{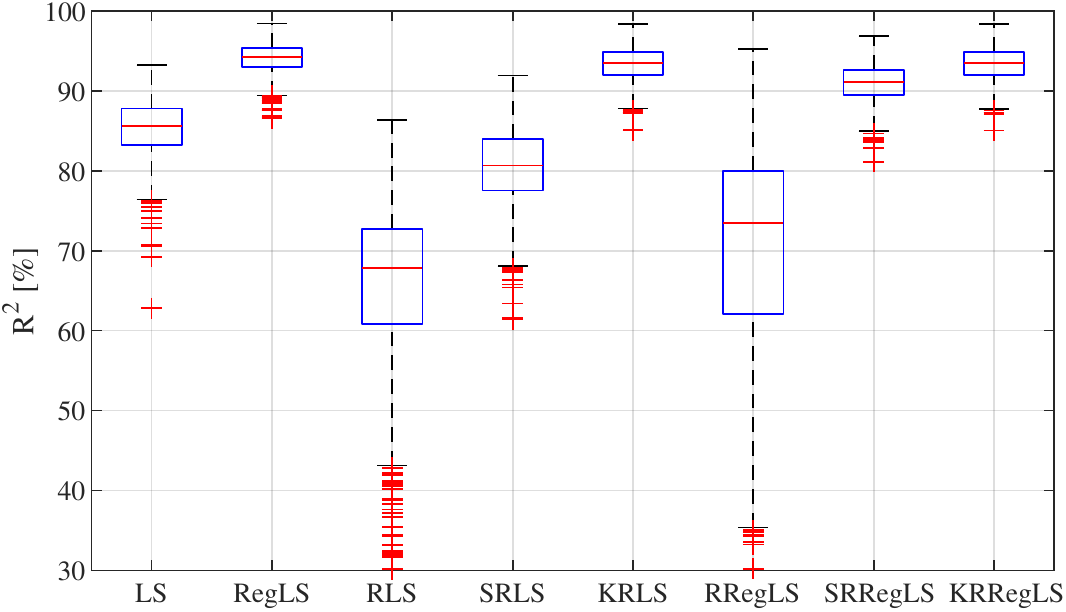}
	\end{center}
	\caption{Comparison of fitting performances.}
	\label{fig:box1}
\end{figure}

\section{Discussion}
Based on the results of Monte Carlo experiments, we have the following observations:

\begin{itemize}
	\item From Figure \ref{fig:hist1}, we observe similar behaviors for RegLS and KRLS, which is expected due to Theorem \ref{thm:RegLS_is_KRLS}. Moreover, the resulting variances are considerably smaller comparing to LS case
	demonstrating the robustness of the RegLS and KRLS approaches. 
	
	\item Note from Figure \ref{fig:hist1}, Figure \ref{fig:box1} and Table \ref{tbl:mc} that na\"ive formulations of robust optimization do not necessarily improve the robustness with respect to the input measurement disturbance.
	Indeed, the performance of RLS and RRegLS is significantly worse than LS. This  can be explained by the nature of robust optimization which attempts to minimize the worst case cost. 
	Accordingly, the shape and the structure of uncertainty set plays a major role in the performance of estimation approaches formulated as a robust optimization.
	This argument is supported by the performance improvements when RLS is compared to SRLS, or RRegLS is compared to SRRegLS.
	
	\item Comparing the estimation performances of RLS, SRLS and KRLS, one may conjecture that the right choices of uncertainty set for the estimation problems concerning system identification should have the same structure as the kernel-based uncertainty set.
	Comparing the performance of RRegLS, SRRegRS, and KRRegLS provides further support for this argument.
	
	\item By comparing RLS and RRegLS, one can see that the choice of uncertainty set plays a more dominant role the regularization term in the estimation method.
	
	\item From Table \ref{tbl:mc}, one can see that the KRLS and KRRegLS methods have the smallest  estimation variance and subsequently, maximal robustness to the input measurement disturbances. This observation highlights the impact of kernel-based uncertainty set. 
	Also, the bias-variance trade-off for these approaches is close to the RegLS method.
	Moreover, one can see that including regularization in KRLS which results in KRRegLS does not improve the estimation performance significantly, at least for the current setting where the output measurement is noiseless.
\end{itemize}

\begin{table}[t]
	\renewcommand{\arraystretch}{1.0}
	\centerline{
		\begin{tabular}{lccccc}
			\hline&
			\textbf{$\text{Bias}$}&\textbf{$\text{Var}$}& 
			\textbf{$\text{MSE}$}&\textbf{Unc.}&\textbf{Reg.}\\&
			{$[\times 10^{4}]$}&{$[\times 10^{4}]$}&{$[\times 10^{4}]$}&\textbf{Set}&\\\hline 
			{LS}		&0.56	&22.0	&22.3	&-&-\\
			{RegLS}		&0.31	&\textbf{3.67}	&\textbf{3.77}	&-&$\checkmark$\\   
			{RLS}  		&10.66	&23.7	&137.4	&standard&-\\ 	
			{SRLS} 		&4.16	&22.5	&39.8	&structured&-\\ 
			{KRLS} 		&1.68	&\textbf{1.91}	&\textbf{4.74}	&kernel-based&-\\ 
			{RRegLS}	&10.35	&47.3	&154.4	&standard&$\checkmark$\\   
			{SRRegLS}	&1.76	&5.33	&8.52&structured&$\checkmark$\\  
			{KRRegLS}	&1.70	&\textbf{1.90}	&\textbf{4.79}	&kernel-based&$\checkmark$\\  
			\hline
		\end{tabular}
	}
	\caption{Statistics for different methods.}
	\label{tbl:mc}
\end{table}

\section{Conclusions}\label{sec:con}
In this paper, we have shown a novel feature of the kernel-based system identification method concerning robustness to the input disturbances. We have proved that the regularized kernel-based approach can be reformulated as a robust least-squares problem with an uncertainty set defined based on the kernel-matrix. Using extensive numerical experiments and comparisons, we have studied the nature of this new uncertainty set. It has been verified that the robust least square identification approach with the kernel-based uncertainty set is robust with respect to input disturbances and retains other features of the kernel-based approach as well. 

\appendix
\section{Appendix} 
\subsection{Duality for Second-Order Cone Programming}
\label{sec:apendix:SOCP_dual}
Consider the following second-order cone programming
\begin{equation}\label{eqn:app_SOCP}
\begin{split}
\minOp_{\vcx\in\Rbb^n}
& \quad
\vcc^\tr \vcx
\\
\mathrm{s.t.}\
& \quad
\|\mxA \vcx-\vcd\| \le \vca^\tr \vcx,
\\
& \quad
\|\mxB \vcx\| \le \vcb^\tr \vcx,
\end{split} \tag{P}
\end{equation}
where $n,m,k$ are positive integers, $\vca,\vcb,\vc{c}\in\Rbb^n$, $\mxA\in\Rbb^{m\times n}$, $\mxB\in\Rbb^{k\times n}$, and $\vcd\in\Rbb^{m}$.
Then, the dual of \eqref{eqn:app_SOCP} is the following convex program
\begin{equation}\label{eqn:app_SOCP_dual}
\begin{split}
\maxOp_{s,t,\vcz,\vcw} 
& \quad
\vcd^\tr \vcz
\\
\mathrm{s.t.}\
& \quad
\mxA^\tr\vcz+\mxB^\tr\vcw + \vca \!\ s+ \vcb \!\ t = \vcc,
\\
& \quad
\|\vcz\|\le s, 	
\\
& \quad
\|\vcw\|\le t,
\\
& \quad
\vcz\in\Rbb^{m},\vcw\in\Rbb^{k}, s,t\in\Rbb.
\end{split} \tag{D}
\end{equation}
Note that when the dual problem \eqref{eqn:app_SOCP_dual} is feasible and bounded, then the primal problem \eqref{eqn:app_SOCP} is also feasible and bounded. Moreover, they attain same optimal values when strong duality holds. 
For more details, see \cite{boyd2004convex}.

\subsection{Supplementary Lemmas}
\label{sec:Lemma_appendix}
\begin{lemma}\label{lem:Cb_le_Cfb}
Let $\vca\in\Rbb^n$ and $\mxD\in\Rbb^{m\times n}$. Then, we have $\|\mxD\vca\|^2\le \|\mxD\|_{\fro}^2\|\vca\|^2$. The equality occurs iff  there exists vector $\vcc\in\Rbb^n$ such that $\mxD = \vcc\!\ \vca^\tr$.
\end{lemma}	
\begin{proof}
We know that 
\begin{equation}\label{eqn:lemma_app_1_1}
\begin{split}
\|\mxD\vca\|^2 
&=
\sum_{i=1}^m(\sum_{j=1}^n d_{ij} a_j)^2
\\&\le
\sum_{i=1}^m(\sum_{j=1}^n d_{ij}^2) \|\vca\|^2
= 
\|\mxD\|_{\fro}^2\|\vca\|^2,
\end{split}
\end{equation}
where the Cauchy-Schwartz inequality is used in the second step.
For any $i=1,\ldots,m$, due to the equality condition for Cauchy-Schwartz theorem, we know that $(\sum_{j=1}^n d_{ij} a_j)^2 = (\sum_{j=1}^n d_{ij}^2) \|\vca\|^2$ iff  there exists $c_i\in\Rbb$ such that we have 
$\begin{bmatrix}d_{i1}&\ \ldots\ &d_{in}\end{bmatrix}= c_i \vca^\tr$.
Therefore, the equality occurs in \eqref{eqn:lemma_app_1_1} iff  there exist $\vcc\in\Rbb^n$ such that $\mxD = \vcc\!\ \vca^\tr$.	
\end{proof}
\begin{lemma}
\label{lem:inner_problem_robust}
Let $\rho>0$, $\vca\in\Rbb^n$, $\vcb\in\Rbb^m$ and $\mxR\in\Rbb^{n\times n}$ be an invertible matrix. 
Then, for the following optimization
\begin{equation}\label{eqn:lemma_app_2_1}
\max_{\substack{\Delta\in \Rbb^{m\times n}\\ \|\Delta\mxR\|_{\fro}\le \rho}}
\ \|\Delta\vca+\vcb\|,
\end{equation}
the optimal value equals $\|\vcb\| +  \|\rho\mxR^{-1}\vca\|$.
\end{lemma}
\begin{proof}
Since for $\vca=\zero$, the result is straightforward, we assume $\vca\ne \zero$. 
Also, let $\vcb \ne \zero$. 
We know that 
\begin{equation}\label{eqn:lemma_app_2_4}
\|\Delta\vca+\vcb\|
\le
\|\Delta\vca\| + \|\vcb\| ,
\end{equation}	
where the equality holds if there exists $\eta>0$ such that 
$\vcb = \eta \Delta \vca$.
Moreover, from Lemma \ref{lem:Cb_le_Cfb}, we have that
\begin{equation}\label{eqn:lemma_app_2_5}
\begin{split}
\|\Delta\vca\| 
&=
\|\Delta\mxR\!\ \mxR^{-1}\vca\|
\\&\le 
\|\Delta\mxR\|_{\fro} \|\mxR^{-1}\vca\|
\\&\le 
\rho \|\mxR^{-1}\vca\|,
\end{split}
\end{equation}	
for any $\Delta\in\Rbb^{m\times n}$ such that $\|\Delta\mxR\|_{\fro}\le \rho$.
The equality occurs in \eqref{eqn:lemma_app_2_5} iff  there exists $\vcc\in\Rbb^m$ such that $\Delta\mxR = \vcc (\mxR^{-1}\vca)^\tr$ and $\|\Delta\mxR\|_{\fro} = \rho$.
From \eqref{eqn:lemma_app_2_4} and \eqref{eqn:lemma_app_2_5}, we have
\begin{equation}\label{eqn:lemma_app_2_6}
\|\Delta\vca+\vcb\|
\le 
\Big(
\rho \|\mxR^{-1}\vca\|+ \|\vcb\|\Big)^2,
\end{equation}
for any $\Delta\in \Rbb^{m\times n}$ such that $\|\Delta\mxR\|_{\fro}\le \rho$.
Note that the right-hand side in \eqref{eqn:lemma_app_2_6} does not depend on $\Delta$ and therefore, it is an upper bound for \eqref{eqn:lemma_app_2_1}.
If the equality holds in \eqref{eqn:lemma_app_2_6}, for a given $\Delta^*$, the  equality conditions mentioned above should be satisfied. More precisely, we need to have  $\Delta^* = \vcc(\mxR^{-1}\vca)^\tr\mxR^{-1}$, for some $\vcc\in\Rbb^m$ such that $\|\vcc\|\|\mxR^{-1}\vca\|= \rho$ 
and there exists $\eta>0$ such that
\begin{equation}
\vcb 
=
\eta\ \! \vcc\ \!  (\mxR^{-1}\vca)^\tr\mxR^{-1}\vca
=
\eta\ \!  \|\mxR^{-1}\vca\|^2\ \! \vcc.
\end{equation}
Consequently, these equality holds iff   
\begin{equation}
\vcc = \frac{\rho\!\ \vcb}
{\|\vcb\|\ \! \|\mxR^{-1}\vca\|},\quad 
\eta = \frac{\rho\!\  \|\vcb\| }{\|\mxR^{-1}\vca\|}.	
\end{equation}
Therefore, \eqref{eqn:lemma_app_2_1} has a unique solution $\Delta^*$ defined as
\begin{equation}\label{eqn:lemma_app_2_2_proof}
	\Delta^* := 
	\frac{\rho\!\  \vcb\!\ (\mxR^{-1}\vca)^\tr\mxR^{-1}}{\|\vcb\|\ \! \|\mxR^{-1}\vca\|}.
\end{equation}
Moreover, the optimal value of \eqref{eqn:lemma_app_2_1} equals to the right-hand side of \eqref{eqn:lemma_app_2_6}. For the case $\vcb = \zero$, due to inequality \eqref{eqn:lemma_app_2_5}, we know that
\begin{equation}\label{eqn:lemma_app_2_7}
\|\Delta\vca+\vcb\|  = \|\Delta\vca\| 
\le 
\rho \|\mxR^{-1}\vca\|
=
\rho \|\mxR^{-1}\vca\| + \|\vcb\|,
\end{equation}	
for any $\Delta\in\Rbb^{m\times n}$ such that $\|\Delta\mxR\|_{\fro}\le \rho$. The  equality holds for $\Delta^*$ iff  these exists $\vcc\in\Rbb^m$ such that 
$\Delta^*\mxR = \vcc(\mxR^{-1}\vca)^\tr$ and $\|\Delta^*\mxR\|_{\fro}=\rho$, or equivalently 
$\Delta^* = \vcc(\mxR^{-1}\vca)^\tr\mxR^{-1}$
and
$\|\vcc\| = \frac{\rho}{\|\mxR^{-1}\vca\|}$. 
Replacing $\vcc$ with $\frac{\rho\!\ \vcu}{\|\mxR^{-1}\vca\|}$, where $\vcu$ is a unit vector in $\Rbb^m$, we obtain $\Delta^*$ as 
\begin{equation}
\label{eqn:lemma_app_2_2b_proof}
\Delta^* := 
\frac{\rho\!\  \vcu\!\ (\mxR^{-1}\vca)^\tr\mxR^{-1}}{\|\mxR^{-1}\vca\|}.
\end{equation}
This concludes the proof. 
\end{proof}

\bibliographystyle{IEEEtran}
\bibliography{mybib}
\end{document}